\documentclass{amsart}

\newtheorem{theorem}{Theorem}[section]
\newtheorem{corollary}[theorem]{Corollary}
\newtheorem{lemma}[theorem]{Lemma}
\newtheorem{proposition}[theorem]{Proposition}

\newtheorem*{lemma1*}{Theorem \ref{thm0}}
\newtheorem*{lemma2*}{Theorem \ref{thm-cHw-1}}
\newtheorem*{lemma3*}{Theorem \ref{thm1}}
\theoremstyle{definition}
\newtheorem{definition}[theorem]{Definition}
\theoremstyle{remark}

\title{Simultaneous Sequential Compactness}

\author[Resteghini]{Sirio Resteghini}
\address{Dipartimento di Matematica, Università di Pisa, Largo Bruno Pontecorvo, 5, 56127 Pisa, Italia}
\email{sirio.resteghini@phd.unipi.it}

\author[Straffelini]{Cesare Straffelini}
 \address
         {Departament de Matemàtiques i Informàtica, Universitat de Barcelona, Gran Via de les Corts Catalanes, 585, 08007 Barcelona, Catalunya; Dipartimento di Matematica, Università degli Studi di Trento, Via Sommarive, 14, 38123 Povo (Trento), Italia}
\email{straffelini@ub.edu \& cesare.straffelini@unitn.it}

\usepackage{latexsym,amsfonts,amsmath,amssymb,mathrsfs}
\usepackage{hyperref}
\usepackage{relsize}
\usepackage[cmyk,svgnames,dvipsnames]{xcolor}
\usepackage{tikz}
\usepackage{comment}
\usetikzlibrary{arrows,arrows.meta,petri,topaths,positioning,shapes,shapes.misc,patterns,calc,decorations.pathreplacing,hobby}
\pgfdeclarelayer{foreground}
\pgfdeclarelayer{background}
\pgfdeclarelayer{boardarrows}
\pgfdeclarelayer{boardgrid}
\pgfdeclarelayer{boardshades}
\pgfsetlayers{boardshades,boardgrid,boardarrows,background,main,foreground}
\usepackage{wrapfig} 
\usepackage{float}
\usepackage[utf8]{inputenc}
\RequirePackage{doi}

\usepackage{enumitem}

\renewcommand{\le}{\leqslant}
\renewcommand{\ge}{\geqslant}

\usepackage{enumitem}

\begin{document}

\date{\today}

\begin{abstract}
A set of sequences is said to converge simultaneously if there exists an infinite subset $H$ of the index set $\omega$ such that all sequences converge when restricted to $H$. We discuss simultaneous convergence of sequences in the same or in different sequentially compact spaces; we link the results for different spaces to ones for the same space; we show that simultaneous convergence happens for less than $\mathfrak s$ sequences in spaces with weight bounded by $\mathfrak s$ and for less than $\mathfrak h$ sequences in general; we show a slight generalisation of these results in the context of Hausdorff spaces; and finally we investigate their optimality.
\end{abstract}

\maketitle

\section{Introduction and Preliminary Results}

Given a topological space $X$ and a sequence $x=\langle x_n:n<\omega\rangle$ of points in $X$, where the index set $\omega$ is the least infinite ordinal, we say that $x$ converges to some point $x_\infty\in X$ if, for every open neighbourhood $U$ of $x_\infty$ in $X$, there exists a natural number $n_U<\omega$ such that $x_n\in U$ for all $n>n_U$. If $H$ is an infinite subset of $\omega$, we say that $x$ converges along $H$ if the subsequence $x\upharpoonright H=\langle x_n:n\in H\rangle$ converges.\medskip

Our aim is to study the following problem: given several sequences (which may take values in different topological spaces), can we find an infinite subset $H$ of $\omega$ such that all these sequences converge along $H$? We seek an answer that may depend both on the number of sequences and on the topological properties of spaces under consideration.\medskip

 Let $[\omega]^\omega$ denote the set of all infinite subsets of $\omega$. For two elements $H_0$ and $H_1$ in $[\omega]^\omega$, we write $H_0=^\ast H_1$ if their symmetric difference is finite (i.e., they differ only on a finite set), and $H_0\subseteq^\ast H_1$ if there exists $H'\subseteq H_1$ with $H_0=^\ast H'$.\medskip
 
 If $H_0=^\ast H_1$, then $x\upharpoonright H_0$ converges if and only if $x\upharpoonright H_1$ does. On the other hand, if $H_0\subseteq^\ast H_1$, then $x\upharpoonright H_0$ converges if (but not necessarily only if) $x\upharpoonright H_1$ converges.
 
\begin{definition} If $\lambda$ is a cardinal, a space $X$ is $\lambda$-\emph{simultaneously sequentially compact} ($\lambda$-ssc for short) if, whenever we are given $\lambda$-many sequences having values in $X$, we can find $H\in[\omega]^\omega$ such that each of those sequences converges along $H$.\end{definition}
  
 Note that if $\lambda\le\kappa$ and a space $X$ is $\kappa$-ssc, then $X$ is $\lambda$-ssc. Moreover, the definition of  ``$1$-ssc'' coincides with the standard definition of sequential compactness.\medskip
 
 Sequential compactness already implies $\aleph_0$-simultaneous sequential compactness, a fact that can be proved via a diagonalisation argument. For sake of completeness, we provide a proof below.

\begin{lemma}\label{1}
A topological space is sequentially compact if and only if it is $\aleph_0$-ssc.
\end{lemma}

\begin{proof}
First, assume $X$ is $\aleph_0$-ssc. Then it is $1$-ssc, and thus sequentially compact.\medskip

Now, assume $X$ is sequentially compact. Let $x^k$ be sequences with values in $X$, for $k<\omega$. By sequential compactness, there exists $H_0\in[\omega]^\omega$ such that $x^0\upharpoonright H_0$ converges. Next, consider the sequence $x^1\upharpoonright H_1$. Since $X$ is sequentially compact, there is $H_1\subseteq H_0$ in $[\omega]^\omega$ such that $x^1\upharpoonright H_1$ converges. By a previous remark, we can add finitely many elements to $H_1$ without losing convergence, so we assume that $H_0$ and $H_1$ have the same least element.\medskip

Proceed by induction. Suppose we have sets $H_i$ for $i<k$ such that $H_{i+1}\subseteq H_i$, and that $H_i$ and $H_{i+1}$ agree on the first $i+1$ elements, with each subsequence $x^i\upharpoonright H_i$ converging. 
Consider $x^k\upharpoonright H_{k-1}$. Since $X$ is sequentially compact, there exists $H_{k}\subseteq H_{k-1}$ such that $x^k\upharpoonright H_k$ converges. We may assume that $H_{k-1}$ and $H_k$ agree on their first $k$ elements.\medskip

Let $H:=\bigcap_{k<\omega} H_k$. This set is infinite because for all $k<\omega$ the first $k$ elements of $H_k$ belong to $H$. Since $H\subseteq H_k$ for all $k<\omega$, every sequence $x^k\upharpoonright H$ converges, as required.
\end{proof}

This work is then structured as follows. In Section \ref{s2}, we generalise a result by Andreas Blass from \cite{blass}, showing the first of our main results:

\begin{lemma1*}
If $\lambda<\mathfrak s$, any sequentially compact space $X$ with $w(X)<\mathfrak s$ is $\lambda$-ssc.
\end{lemma1*}

Here, $\mathfrak s$ denotes the splitting number.\medskip

In Section \ref{s3}, we consider the case of Hausdorff spaces. We define a cardinal invariant, called \emph{countable Hausdorff weight} ($\mathrm{cHw}$), which we use in place of the weight to improve on Theorem \ref{thm0}:

\begin{lemma2*}
    If $\lambda<\mathfrak s$, any sequentially compact Hausdorff space $X$ with $\mathrm{cHw}(X)<\mathfrak s$ is $\lambda$-ssc.
\end{lemma2*}

In Section \ref{sect-different-spaces}, we address the problem of sequences in different topological spaces. There, both the general results from Section \ref{s2} and the specific results for Hausdorff spaces from Section \ref{s3} are generalised to this new context, showing that they essentially continue to hold, provided that all sequences come from spaces with comparable properties.\medskip

Finally, in Section \ref{s5} we prove the following result, which does not rely on any restriction on the cardinal invariants of the space:

\begin{lemma3*}
If $\lambda<\mathfrak h$, any sequentially compact space $X$ is $\lambda$-ssc.
\end{lemma3*}

Here, $\mathfrak h$ denotes the distributivity number. At the end, we propose some further questions regarding the optimality of some of our results.

\section{Results based on the Splitting Number}\label{s2}

Lemma \ref{1} provides a lower bound for the extent of simultaneous sequential compactness exhibited by a sequentially compact space. We now turn to the question of finding an upper bound, which, unlike the lower bound, will depend on additional topological properties of the space. Indeed, in the extremely trivial case where a space only has one element, it is easy to show that it is $\lambda$-ssc for all cardinals $\lambda$.

\begin{definition}We say that a topological space $X$ is ``nontrivial enough to distinguish two points'' (in the following, we just write nontrivial) if in $X$ there are two elements $a$ and $b$ such that the closure of $\{a\}$ and the closure of $\{b\}$ are disjoint.\end{definition}

For a nontrivial space $X$, there is indeed an upper bound for the possible values of $\lambda$ for which $X$ could be $\lambda$-ssc. This bound is given by the splitting number $\mathfrak s$.

\begin{definition}
A family $\mathcal S\subseteq[\omega]^\omega$ is said to be a \emph{splitting family} if for every $A\in[\omega]^\omega$ there is some $S\in\mathcal S$ such that both $A\cap S$ and $A-S$ are infinite sets. The \emph{splitting number} $\mathfrak s$ is then defined to be the least cardinality of a splitting family.
\end{definition}

It holds $\aleph_1\le \mathfrak s\le 2^{\aleph_0}$, and it is consistent relative to \textsf{ZFC} that $\aleph_1<\mathfrak s< 2^{\aleph_0}$. For more informations about the splitting number, we refer the interested reader to Andreas Blass's handbook chapter \cite{blass}. There, one may also find the proof of the statement ``the discrete space $\{0,1\}$ is not $\mathfrak s$-ssc'', which, as can be easily seen, can be readapted to prove also the apparently stronger Lemma \ref{2}.

\begin{lemma}\label{2}
If $X$ is a nontrivial topological space, $X$ is not $\mathfrak s$-ssc.
\end{lemma}

\begin{proof}Let $a$ and $b$ be in $X$ such that the closures of $\{a\}$ and $\{b\}$ are disjoint.\medskip

Let $\mathcal S=\langle S_\alpha:\alpha<\mathfrak s\rangle$ be a splitting family of size $\mathfrak s$. For $\alpha<\mathfrak s$, let $x^\alpha$ be
$$x^\alpha_n:=\begin{cases}a\ \text{if}\ n\in S_\alpha;\\ b\ \text{if}\ n\not\in S_\alpha.\end{cases}$$

Since $\mathcal S$ is a splitting family, for every $H\in[\omega]^\omega$ there is $\alpha<\mathfrak s$ such that $H\cap S_\alpha$ and $H-S_\alpha$ are both infinite. Therefore, the sequence $x^\alpha\upharpoonright H$ takes both values $a$ and $b$ infinitely often, hence it cannot converge, because it is frequently outside a neighbourhood of $a$ (the complementary of the closure of $\{b\}$) and outside a neighbourhood of $b$ (the complementary of the closure of $\{a\}$).\medskip

Hence there is a family $\langle x^\alpha: \alpha<\mathfrak s\rangle$ of size $\mathfrak s$ of sequences with values in $X$ with no $H\in[\omega]^\omega$ along which all sequences converge, so $X$ is not $\mathfrak s$-ssc, as desired.\end{proof}

Therefore, in the case of nontrivial spaces, the best result we can aim for is them to be $\lambda$-ssc for $\lambda<\mathfrak s$. As proved in \cite{blass}, for the discrete space $\{0,1\}$ this holds. We show now the proof of this fact using our notations and definitions.

\begin{lemma}\label{3}
If $\lambda<\mathfrak s$, the discrete space $\{0,1\}$ is $\lambda$-ssc.
\end{lemma}

\begin{proof}Let $\lambda<\mathfrak s$ and let $x^\alpha$ be a sequence with values in $\{0,1\}$, for $\alpha<\lambda$.\medskip

For ordinals $\alpha<\lambda$, let $A_\alpha:=\{n\in\omega: x^\alpha_n=1\}$ and $B_\alpha:=\{n\in\omega: x^\alpha_n=0\}$. If $A_\alpha$ is infinite, let $S_\alpha:=A_\alpha$; otherwise let $S_\alpha:=B_\alpha$. Now consider the family $\mathcal F:=\{S_\alpha:\alpha<\lambda\}\subseteq[\omega]^\omega$. Since $\lambda<\mathfrak s$, $\mathcal F$ cannot be a splitting family, hence there is some $H\in[\omega]^\omega$ such that, for all $\alpha<\lambda$, either $H\cap S_\alpha$ or $H-S_\alpha$ is finite.\medskip

This means that, for all $\alpha<\lambda$, either $H\cap A_\alpha$ or $H\cap B_\alpha$ is finite. If $H\cap A_\alpha$ is finite, then $x^\alpha\upharpoonright H$ converges to $0$; if $H\cap B_\alpha$ is finite, $x^\alpha\upharpoonright H$ converges to $1$.
\end{proof}

This proof is using some particular properties of the space $\{0,1\}$, hence it can not be as easily generalised to all nontrivial sequentially compact spaces.\medskip

In the already cited chapter \cite{blass}, Blass proves also the following Lemma \ref{blass}.

\begin{lemma}\label{blass}
If $\lambda<\mathfrak s$, the compact unit interval $[0,1]\subseteq\mathbb R$ is $\lambda$-ssc.
\end{lemma}

We are not going to prove Lemma \ref{blass} because it will be a corollary to Theorem \ref{thm0}. Though, we can notice that Blass's proof is indeed showing something stronger than Lemma \ref{blass}: the same exact argument works for any compact subset of any separable metric space.\medskip

Theorem \ref{thm0} is even stronger than this, since it implies that any compact subset of any metric space is $\lambda$-ssc for $\lambda<\mathfrak s$. Let us recall an important topological notion.

\begin{definition}Given a topological space $X$, its weight is defined as
$$ w(X):=\aleph_0+\min\{\left|\mathcal B\right|: \mathcal B\ \text{is a basis for the topology on}\ X\}.$$\end{definition}

\begin{theorem}\label{thm0}
If $\lambda<\mathfrak s$, any sequentially compact space $X$ with $w(X)<\mathfrak s$ is $\lambda$-ssc.
\end{theorem}

\begin{proof}
Let $\lambda<\mathfrak s$ and consider $\lambda$-many sequences $x^\alpha$ for $\alpha<\lambda$ with values in $X$.\medskip

Let $\mathcal B$ be a basis for the topology on $X$ with $\left|\mathcal B\right|<\mathfrak s$. Let $\kappa:=\left|\mathcal B\right|$. Consider an enumeration $\langle B^{\beta}: \beta<\kappa\rangle$ of $\mathcal B$. For all $\alpha<\lambda$ and $\beta<\kappa$ let us define

$$y^{\alpha,\beta}_n:=\begin{cases} 1& \text{if}\ x^\alpha_n\in B^{\beta};\\ 0& \text{if}\ x^\alpha_n\not\in B^{\beta}.\end{cases}$$

The set $\{y^{\alpha,\beta}: \alpha<\lambda, \beta<\kappa\}$ is a set of $\lambda\kappa$-many sequences with values in $\{0,1\}$, so since $\lambda\kappa<\mathfrak s$ there is by Lemma \ref{3} some $H\in[\omega]^\omega$ such that $y^{\alpha,\beta}\upharpoonright H$ converges for all $\alpha<\lambda$ and all $\beta<\kappa$. We want to prove that $x^\alpha\upharpoonright H$ converges, for all $\alpha<\lambda$.\medskip

Fix $\alpha<\lambda$. Since $X$ is sequentially compact, there is $H'\in[\omega]^\omega$ such that $H'\subseteq H$ and the sequence $x^\alpha\upharpoonright H'$ converges to some $x^\alpha_\infty$.  For all $\beta<\kappa$, if $x^\alpha_\infty\in B^\beta$, then the sequence $x^\alpha\upharpoonright H'$ from some point onwards is always inside $B^\beta$, hence $y^{\alpha,\beta}\upharpoonright H'$ converges to $1$, so since we know that $y^{\alpha,\beta}\upharpoonright H$ converges, and $H'\subseteq H$, the sequence $y^{\alpha,\beta}\upharpoonright H$ also converges to $1$. This implies that $X^\alpha\upharpoonright H$ lies definitively inside every open neighbourhood of $x^\alpha_\infty$, and thus it converges to $x^\alpha_\infty$, as desired.
\end{proof}

A proof similar to the one of Theorem \ref{thm0} implies the same result of $\lambda$-ssc for $\lambda<\mathfrak s$ for compact spaces of weight less than $\mathfrak s$. Recall that, in general, there is no direct correlation between topological compactness and sequential compactness.

\begin{theorem}\label{thm0bis}
Every compact space $X$ with $w(X)<\mathfrak s$ is $\lambda$-ssc for all $\lambda<\mathfrak s$.
\end{theorem}

\begin{proof}
Let $\lambda$, $x^\alpha$, $\mathcal B$, $\kappa$, $y^{\alpha,\beta}$ and $H$ be as in the proof of Theorem \ref{thm0}. Our goal is to prove, as in the proof of Theorem \ref{thm0}, that $x^\alpha\upharpoonright H$ converges, for all $\alpha<\lambda$.\medskip

Assume by contradiction that $x^\alpha\upharpoonright H$ does not converge, for some $\alpha<\lambda$. This means that for all $x\in X$ there is a $\beta(x)<\kappa$ such that $x\in B^{\beta(x)}$ but the sequence $x^\alpha\upharpoonright H$ is infinitely many times outside of $B^{\beta(x)}$, so the sequence $y^{\alpha,\beta(x)}\upharpoonright H$ is infinitely many times $0$, hence since it converges, it has to converge to $0$.\medskip

But now $\bigcup_{x\in X} B^{\beta(x)}$ is an open cover of $X$, hence we can extract a finite subcover given by $X=B^{\beta(x_1)}\cup\ldots\cup B^{\beta(x_k)}$ for some $\{x_1,\ldots,x_k\}\subseteq X$. Since $y^{\alpha,\beta(x)}\upharpoonright H$ converges to $0$ for all $x\in X$, the sequence $x^\alpha\upharpoonright H$ is definitively outside each $B^{\beta(x)}$, so it is definitively outside $B^{\beta(x_1)}\cup\ldots\cup B^{\beta(x_k)}$, that is $X$, a contradiction.
\end{proof}

Since $1<\mathfrak s$, as a corollary of Theorem \ref{thm0bis} we get back the following Corollary \ref{vandouwen}, which is a well-known result that can be found for example in Eric Van Douwen's handbook chapter \cite{vandouwen}. Notice how Theorem \ref{thm0bis} is a direct generalisation of it.

\begin{corollary}\label{vandouwen}
Every compact space $X$ with $w(X)<\mathfrak s$ is sequentially compact.
\end{corollary}

Theorem \ref{thm0} generalises several classical results in analysis, thus it has the potential to be employed in practical applications. For instance, the following Corollary \ref{bba} extends the Banach–Bourbaki–Alaoglu theorem, showing that simultaneous weak-$\star$ sequential compactness holds for small families of functionals.

\begin{corollary}\label{bba}
Let $\mathbb E$ be a separable normed vector space. If $\lambda<\mathfrak s$, then the unit ball of its topological dual $\mathbb E'$ is $\lambda$-simultaneously weakly-$\star$ sequentially compact.
\end{corollary}

\section{Hausdorff Spaces}\label{s3}

{Now we turn our attention to Hausdorff spaces, in order to show an improvement to Theorem \ref{thm0} using a different kind of cardinal invariant of the topological space.}\medskip

If $X$ is a topological space and $Y\subseteq X$ is a subset, we denote by $\mathcal L(Y)$ the set of limit points of $Y$ in the topology of $X$, i.e., the set of all points $x_\infty\in X$ such that there exists a sequence in $Y$ that converges to $x_\infty$. This may differ from the sequential closure of $Y$, $\mathrm{SeqCl}(Y)$, by which we mean the smallest sequentially closed superset of $Y$, because $\mathcal L(Y)$ might not be sequentially closed.

\begin{definition}
  Let $X$ be a topological space, $Y\subseteq X$. A \emph{sequentially separating limit cover} for $Y$ is a family $\mathcal C_Y$ of open subsets of $\mathcal L(Y)$ such that for each $p\in\mathcal L(Y)$ and each $q\in X$ with $q\neq p$ there is $U\in\mathcal C_Y$ with $p\in U$ and $q\not\in \mathrm{SeqCl}(U)$.
\end{definition}

\begin{definition}\label{def-cHw}
  The \emph{countable Hausdorff weight} of an Hausdorff space $X$, denoted by $\mathrm{cHw}(X)$, is the least infinite cardinal $\kappa$ such that each countable subset $Y\subseteq X$ admits a sequentially separating limit cover $\mathcal C_Y$ such that $|\mathcal C_Y|\le \kappa$.
\end{definition}

If $X$ is not an Hausdorff space, there may not exist sequentially separating limit covers for some countable $Y\subseteq X$, and thus $\text{cHw}(X)$ may not be defined. {Instead, for $X$ an Hausdorff space, its countable Hausdorff weight is always defined and is at most $w(X)$, as the following Lemma \ref{rmk-cHw-w} shows.}

\begin{lemma}\label{rmk-cHw-w}
  For each Hausdorff space $X$ we have $\mathrm{cHw}(X)\le w(X)$.
\end{lemma}

\begin{proof}
If $\mathcal B$ is a basis for $X$ with $|\mathcal B|=w(X)$ and $Y\subseteq X$ is countable, the family \[\{U\cap\mathcal L(Y) : U\in\mathcal B\}\] is a sequentially separating limit cover for $Y$ of cardinality at most $w(X)$.
\end{proof}

One interesting remark, shown to us by Angelo Bella, is that while the weight of a topological space is not bounded by the cardinality of the space (not even when the space is Hausdorff), we can prove an inequality for the countable Hausdorff weight:

\begin{lemma}
    For $X$ a topological Hausdorff space, it holds $\mathrm{cHw}(X)\le |X|$.
\end{lemma}

\begin{proof}
Let $Y\subseteq X$ be countable. In order to prove the lemma, it is enough to show that $Y$ has a sequentially separating limit cover of size at most $|X|$. For each $p\neq q$ in $X$, let us fix, by the Hausdorff property, open subsets $U_{p}^q\subseteq X$ and $U_q^p\subseteq X$ such that $p\in U_p^q$, $q\in U_q^p$ and $U_p^q\cap U_q^p=\emptyset$. Let now
\[\mathcal C:=\{U_p^q\cap \mathcal L(Y): p\in \mathcal L(Y), q\in X, p\neq q\}.\]
Since $\mathrm{SeqCl}(U_p^q\cap \mathcal L(Y))\subseteq \mathcal L(Y)-U_q^p$, this is a sequentially separating limit cover for $Y$. Moreover, it holds $|\mathcal C|\le |X|^2=|X|$, as desired.
\end{proof}

We now state our improvement to Theorem \ref{thm0} for Hausdorff spaces.

\begin{theorem}\label{thm-cHw-1}
    If $\lambda<\mathfrak s$, any sequentially compact Hausdorff space $X$ with $\mathrm{cHw}(X)<\mathfrak s$ is $\lambda$-ssc.
\end{theorem}
\begin{proof}
    For each $\alpha<\lambda$, let $x^{\alpha}=\langle x^{\alpha}_n:n<\omega\rangle$ be a sequence of values in $X$, and let $\mathcal C_\alpha$ be a sequentially separating limit cover for {the countable set} $\{x^{\alpha}_n:n<\omega\}$ such that $|\mathcal C_\alpha|\le\mathrm{cHw}(X)$. For each $\alpha<\lambda$ and $U\in\mathcal C_\alpha$, we define a sequence $s^{\alpha,U}$ with values in $\{0,1\}$ as follows:
    \[s^{\alpha,U}_n=\begin{cases}
        1 &\mathrm{if}\ x^{\alpha}_n\in U;\\
        0 &\mathrm{if}\ x^{\alpha}_n\notin U.
    \end{cases}\]
    Notice that, since $\lambda<\mathfrak s$ and $|\mathcal C_\alpha|\le\mathrm{cHw}(X)<\mathfrak s$ for each $\alpha<\lambda$, the set of pairs $\langle\alpha,U\rangle$ with $\alpha<\lambda$ and $U\in\mathcal C_\alpha$ has size less than $\mathfrak s$. Therefore, by Lemma \ref{3}, there is $H\in [\omega]^\omega$ such that, for each $\alpha<\lambda$ and $U\in\mathcal C_\alpha$, the sequence $s^{\alpha,U}\upharpoonright H$ converges. We now aim to prove that the sequence $x^{\alpha}\upharpoonright H$ converges for each $\alpha<\lambda$.\medskip

    Fix $\alpha<\lambda$. Since $X$ is sequentially compact, there are subsets $B,B'\subseteq H$ such that $x^{\alpha}\upharpoonright B$ and $x^{\alpha}\upharpoonright B'$ both converge, to some points $x^{\alpha}_B$ and $x^{\alpha}_{B'}$, respectively. Suppose by contradiction that $x^{\alpha}_B\ne x^{\alpha}_{B'}$. Then, by definition of sequentially separating limit cover, there exists $V\in \mathcal C_\alpha$ with $x^{\alpha}_B\in V$ and $x^{\alpha}_{B'}\notin\mathrm{SeqCl}(V)$.\medskip
    
    We know that $s^{\alpha,V}\upharpoonright H$ converges, and since $x^{\alpha}_B\in V$ we know that $s^{\alpha,V}\upharpoonright H$ must converge to 1. Therefore, $s^{\alpha,V}\upharpoonright B'$ also converges to 1, so $x^{\alpha}_n\in V$ for infinitely many $n\in B'$, which contradicts  $x^{\alpha}_{B'}\notin\mathrm{SeqCl}(V)$. This yields that $x^{\alpha}_B = x^{\alpha}_{B'}$ {for any choice of $B$ and $B'$ as above}, so we can denote $x^{\alpha}_B$ simply by $x^{\alpha}_\infty$.\medskip
    
    Notice that, for each $H'\subseteq H$, there exists $H''\subseteq H'$ such that $x^{\alpha}\upharpoonright H''$ converges, and we have proven that it must converge to $x^{\alpha}_\infty$. Next we employ a result that is true in any topological space, and is sometimes called \emph{Urysohn's property}: a sequence $x$ converges to a point $p$ if and only if every subsequence of $x$ has a further subsequence converging to $p$. This, applied to our context, yields that $x^{\alpha}\upharpoonright H$ converges to $x^{\alpha}_\infty$, which proves the statement.
\end{proof}

We can further refine Theorem \ref{thm-cHw-1}, introducing the following:

\begin{definition}
    We say that a Hausdorff space $X$ \emph{reaches} its countable Hausdorff weight if there exists a countable subset $Y\subseteq X$ such that each sequentially separating limit cover for $Y$ has cardinality at least $\mathrm{cHw}(X)$. 
\end{definition}

\begin{theorem}\label{thm-cHw-2}
Let $X$ be a sequentially compact Hausdorff space with $\mathrm{cHw}(X)=\mathfrak s$, not reaching its countable Hausdorff weight. Then, $X$ is $\lambda$-ssc for all $\lambda<\mathrm{cof}(\mathfrak s)$.
\end{theorem}
\begin{proof}
    The proof is identical to the proof of Theorem \ref{thm-cHw-1}, except that now the set of pairs $\langle \alpha,U\rangle$ with $\alpha<\lambda$ and $U\in\mathcal C_\alpha$ has size less than $\mathfrak s$ simply because $\lambda<\mathrm{cof}(\mathfrak s)$ and $|\mathcal C_\alpha|<\mathrm{cHw}(X)=\mathfrak s$ for all $\alpha<\lambda$. 
\end{proof}

Notice that Lemma \ref{rmk-cHw-w} guarantees that the assumption of Theorem \ref{thm-cHw-1} is indeed a relaxation of the assumption of Theorem \ref{thm0}, for Hausdorff spaces. To prove that this is a meaningful improvement, we show an example where $\text{cHw}(X)< \mathfrak s \le w(X)$, using some results from Section \ref{sect-different-spaces}.

\begin{proposition}
    Let $X'$ be a discrete topological space with $|X'|=\kappa\ge\aleph_0$, and let $X=X'\cup\{\Omega\}$ be the Aleksandroff compactification of $X'$. Then, $\mathrm{cHw}(X)=\aleph_0$ and $w(X)=\kappa$. In particular, if $\kappa\ge \mathfrak s$, we have $\text{cHw}(X)< \mathfrak s \le w(X)$.
\end{proposition}
\begin{proof}
    Since every point of $X'$ is open, it is obvious that $w(X')=\kappa$, and Lemma \ref{w-preserved} yields $w(X)=\kappa$. Since $X'$ is the disjoint union of $\kappa$-many points, Theorem \ref{thm-cHw-grosso} yields $\mathrm{cHw}(X)\le\aleph_0$. Definition \ref{def-cHw} implies $\mathrm{cHw}(X)\ge\aleph_0$, hence the statement. 
\end{proof}

\section{Sequences in Different Spaces}\label{sect-different-spaces}

In this section, we examine the --seemingly more general-- case in which the sequences may belong to different topological spaces. Indeed, in Blass's proof of Lemma \ref{blass}, he does not assume that all the sequences lie in the same interval $[0,1]$; rather, he only requires that each sequence belongs to some compact subset of $\mathbb R$. Here we show that the two situations are, in a certain regard, equivalent, hence all results for sequences in different spaces follows as corollaries of the theorems for sequences in the same space.

\begin{definition}\label{sscfamily} If $\lambda$ is a cardinal, a family $\mathcal F$ of topological spaces is $\lambda$-ssc if given $\lambda$-many sequences, each of which with values in one of the topological spaces in $\mathcal F$, there is $H\in[\omega]^\omega$ such that each considered sequence converges along $H$.\end{definition}

Notice that a family $\mathcal F$ is $\lambda$-ssc if and only if every subfamily $\mathcal G\subseteq \mathcal F$ with $|\mathcal G|\le\lambda$ is $\lambda$-ssc. This allows us, in cases where the family $\mathcal F$ is a proper class, to restrict to families that are sets. This justifies some abuses of notation throughout the rest of this section, for example the implication from Theorem \ref{thm0.5new} (which requires $\mathcal F$ to be a set) to Corollary \ref{1+} (which deals with a proper class).\medskip

Let us recall a famous topological construction, that we will use more than once. Given a topological space $X'$ which is not compact, let $\Omega$ be not in $X'$ and define $$X:=X'\cup\{\Omega\}.$$
Consider the topology on $X$ where the open sets are the sets which are either equal to $A$ for some $A\subseteq X'$ open, or to $X-K$ for some $K\subseteq X'$ compact. Then the space $X$, endowed with this topology, is compact and satisfies the following:

\begin{lemma}\label{w-preserved}
    $w(X)=w(X')$.
\end{lemma}

This space $X$ is called the \emph{Aleksandroff compactification} of $X'$. For further information about this standard topological construction, and for proofs of the results just mentioned, we refer the reader to Ryszard Engelking's seminal book on general topology, \cite{engelking}. The next Lemma \ref{prethm0.5new} is of fundamental importance for the presentation of our results.\medskip

We note that here, and later on as well, we speak of the disjoint union of certain spaces. Formally, if $\mathcal F=\langle Y_\alpha:\alpha<\mu\rangle$ is a sequence (or even merely a family, which we may well-order thanks to the axiom of choice) of spaces, not necessarily distinct, the disjoint union of all spaces in $\mathcal F$ is defined as the set \[\textstyle\bigcup_{\alpha<\mu} (\{\alpha\}\times Y_\alpha),\] where the topology is the one generated by taking, for each $\alpha<\mu$, the topology on $\{\alpha\}\times Y_\alpha$ that makes it homeomorphic to $Y_\alpha$ via the natural bijection $\langle \alpha,y\rangle \mapsto y$.

\begin{lemma}\label{prethm0.5new}
    Let $\mathcal F$ be an infinite family of sequentially compact topological spaces, let $X'$ be the disjoint union of all spaces in $\mathcal F$, and let $X=X'\cup\{\Omega\}$ be the Aleksandroff compactification of $X'$. Then, $X$ is sequentially compact.
\end{lemma}
\begin{proof}
Order $\mathcal F$ as $\langle Y_\alpha: \alpha<\mu\rangle$, with the axiom of choice. Let $x$ be a sequence with values in $X$. We need to show that there is $H\in[\omega]^\omega$ such that $x\upharpoonright H$ converges.\medskip

For each $\alpha<\mu$, let $H_\alpha:=\{n<\omega: x_n\in \{\alpha\}\times Y_\alpha\}$. In addition to this, let also $H_\Omega:=\{n<\omega: x_n=\Omega\}$. Suppose there is some $\alpha_0<\mu$ such that $H_{\alpha_0}$ is infinite. Then, $x\upharpoonright H_{\alpha_0}$ is a sequence with values in $\{\alpha_0\}\times Y_{\alpha_0}$, that is isomorphic to $Y_{\alpha_0}$ hence sequentially compact. Thus, $x\upharpoonright H_{\alpha_0}$ admits a converging subsequence, so $x$ as well admits a converging subsequence. Similarly, if $H_\Omega$ is infinite, then $x$ has a subsequence constantly equal to $\Omega$, hence a converging one.\medskip

To finish the proof, assume that $H_\alpha$ is finite for all $\alpha<\mu$, and also $H_\Omega$ is finite. Let us prove that $x$ converges to $\Omega$. Fix any compact subset $K\subseteq X'$, and let \[\mathcal U:=\{K\cap (\{\alpha\}\times Y_\alpha): \alpha<\mu\}.\] Notice that $\mathcal U$ is an open cover of the compact set $K$, and that the elements of $\mathcal U$ are pairwise disjoint. Therefore, $\mathcal U$ is finite, thus the set \[\mathcal F_K:=\{\alpha<\mu: K\cap (\{\alpha\}\times Y_\alpha)\ne\emptyset\}\] is finite, and $H_{K}:=\bigcup_{\alpha\in\mathcal F_K}H_\alpha$ is finite. Therefore, for each $n>\max H_{K}$, we have $x_n\in X-K$. Since the choice of $K$ was arbitrary, and all open neighbourhoods of $\Omega$ are of the form $X-K$ for $K$ compact, this proves that $x$ converges to $\Omega$.
\end{proof}

\begin{theorem}\label{thm0.5new}
    Let $\mathcal F$ be an infinite family of sequentially compact topological spaces, $X'$ the disjoint union of all spaces in $\mathcal F$, and $X=X'\cup\{\Omega\}$ the Aleksandroff compactification of $X'$. Suppose $X$ is $\lambda$-ssc. Then, the family $\mathcal F$ is $\lambda$-ssc.
\end{theorem}
\begin{proof}
For $\alpha<\lambda$, let $X^\alpha\in\mathcal F$ and let $x^\alpha$ be a sequence with values in $X^\alpha$. Consider the sequences $y^\alpha$ for $\alpha<\lambda$, defined simply as $y^\alpha_n:=\langle \alpha, x^\alpha_n\rangle\in X$. These are $\lambda$-many sequences in $X$, hence there is $H\in[\omega]^\omega$ such that every $y^\alpha\upharpoonright H$ converges. We are going to prove that $x^\alpha\upharpoonright H$ converges, for $\alpha<\lambda$.\medskip

The sequence $y^\alpha\upharpoonright H$ never takes value $\Omega$, hence if it converges in $X$ it means it already converged in the topology on $X'$. Since $y^\alpha$ takes values in $\{\alpha\}\times X^\alpha$, the sequence $y^\alpha\upharpoonright H$ converges there, hence the sequence $x^\alpha\upharpoonright H$ converges in $X^\alpha$.
\end{proof}

As corollaries to Lemma \ref{prethm0.5new} and Theorem \ref{thm0.5new}, we get some simple generalisations of Lemmata \ref{1} and \ref{2}.

\begin{corollary}\label{1+}
The family of all sequentially compact spaces is $\aleph_0$-ssc.
\end{corollary}

\begin{corollary}\label{2+}
The family of all sequentially compact spaces is not $\mathfrak s$-ssc.
\end{corollary}

As a generalisation for Theorem \ref{thm0}, one could expect to be possible to prove ``if $\lambda<\mathfrak s$, the family of all sequentially compact spaces of weight less than $\mathfrak s$ is $\lambda$-ssc".\medskip

 Anyway, this can not be inferred from Theorems \ref{thm0} and \ref{thm0.5new}, and --at the time of writing this sentence-- the authors still do not know if it is provable in \textsf{ZFC} alone.  The corollary we can get straight away from Theorems \ref{thm0} and \ref{thm0.5new} is Lemma \ref{thm0+}.  

\begin{lemma}\label{thm0+}
Fix some cardinal $\kappa<\mathfrak s$. Then, if $\lambda<\mathfrak s$, the family of all sequentially compact topological spaces having weight not greater than $\kappa$ is $\lambda$-ssc.
\end{lemma}

\begin{proof}
    For $\alpha<\lambda$, let $X^\alpha$ be a sequentially compact set with $w(X^\alpha)<\kappa$, and let $x^\alpha$ be a sequence with values in $X^\alpha$. Let $\mathcal F:=\{X^\alpha:\alpha<\lambda\}$, and let also $X$ be the Aleksandroff compactification of the disjoint union of all spaces in $\mathcal F$.\medskip
    
    Lemmata \ref{w-preserved} and \ref{prethm0.5new}, together with Theorem \ref{thm0}, yield that $X$ is $\lambda$-ssc, which by Theorem \ref{thm0.5new} implies that the family $\mathcal F$ is $\lambda$-ssc.
\end{proof}

The stronger desired result can be obtained by assuming the regularity of $\mathfrak s$.

\begin{lemma}\label{thm0+regular}Assume that $\mathfrak s$ is regular. If $\lambda<\mathfrak s$, then the family of all sequentially compact topological spaces having weight strictly less than $\mathfrak s$ is $\lambda$-ssc.
\end{lemma}

\begin{proof}
For $\alpha<\lambda$, let $X^\alpha$ be a sequentially compact space with $w(X^\alpha)<\mathfrak s$. Then
$\{w(X^\alpha): \alpha<\lambda\}$ is a set of cardinals below $\mathfrak s$ with size less than $\mathfrak s$. Since $\mathfrak s$ is regular, $\kappa:=\sup\{w(X^\alpha): \alpha<\lambda\}<\mathfrak s$. The statement follows from Lemma \ref{thm0+}.
\end{proof}

 Given a topological space $X$, its \emph{sequential refinement}, which we denote by $\mathrm{seq}(X)$, is the topological space that has the same underlying set as $X$ and whose closed sets are precisely the sequentially closed sets of $X$. It is straightforward to verify that a sequence in $X$ converges to a certain value with respect to the topology of $X$ if and only if it converges to the same value in the topology of $\mathrm{seq}(X)$. This entails that a space $X$ is sequentially compact (or, more generally, $\lambda$-ssc) if and only if $\mathrm{seq}(X)$ is sequentially compact (respectively, $\lambda$-ssc).\medskip
 
 Furthermore, since the topology of $\mathrm{seq}(X)$ is a refinement of the topology of $X$, it follows that if $X$ is Hausdorff, then $\mathrm{seq}(X)$ is also Hausdorff. The fact that $\mathrm{seq}(X)$ refines the topology of $X$ also yields:
 
 \begin{lemma}\label{chwseq}
     If $X$ is a sequentially compact Hausdorff space, \[\mathrm{cHw}(\mathrm{seq}(X))\le\mathrm{cHw}(X).\]
 \end{lemma}

We now introduce a notion similar to the Aleksandroff compactification.

\begin{definition}
   Let $X'$ be a topological space that is not sequentially compact, and let $\Omega\notin X'$. Define $X:= X'\cup\{\Omega\}$. Consider the topology on $X$ where the open sets are the sets which are either equal to $A$ for some $A\subseteq \mathrm{seq}(X')$ open, or to $X-K$ for some $K\subseteq X'$ sequentially compact. We say that the space $X$, endowed with this topology, is the \emph{sequential Aleksandroff compactification} of $X'$. 
\end{definition}

\begin{lemma}\label{seq-compactification-is-seq-compact}
   Let $X'$ be a topological space that is not sequentially compact. Then, the sequential Aleksandroff compactification $X$ of $X'$ is sequentially compact.
\end{lemma}

\begin{proof}
   Let $x$ be a sequence with values in $X$, and assume it has no subsequence with values in $X'$ that converges in $X'$.
   Let $K\subseteq X'$ be sequentially compact. If there exists $H\in[\omega]^\omega$ such that $x\upharpoonright H$ has values in $K$, then $x\upharpoonright H$ admits a converging subsequence, which contradicts our assumption on $x$. Therefore, $x$ lies definitively in $X-K$. Since this holds for each such $K$, we have that $x$ lies definitively in each open neighbourhood of $\Omega$ in $X$, and thus it converges to $\Omega$.
\end{proof}

Thanks to Lemma \ref{seq-compactification-is-seq-compact}, we can prove the exact analogue of Theorem \ref{thm0.5new} for the setting of the sequential Aleksandroff compactification:

\begin{theorem}\label{thm0.5new+}
    Let $\mathcal F$ be an infinite family of sequentially compact spaces, $X'$ the disjoint union of all spaces in $\mathcal F$, and $X=X'\cup\{\Omega\}$ the sequential Aleksandroff compactification of $X'$. Suppose that $X$ is $\lambda$-ssc. Then, the family $\mathcal F$ is $\lambda$-ssc.
\end{theorem}

\begin{proof}
 The proof is exactly the same as for Theorem \ref{thm0.5new}. Notice, in particular, that \[\textstyle\mathrm{seq}(\bigcup_{\alpha<\mu} \{\alpha\}\times Y_\alpha)=\bigcup_{\alpha<\mu} \{\alpha\}\times \mathrm{seq}(Y_\alpha).\qedhere\]
\end{proof}

Next, we want to apply Theorem \ref{thm0.5new+} to the setting of Hausdorff spaces with bounded countable Hausdorff weight, to show some variations of Lemmata \ref{thm0+} and \ref{thm0+regular}, that will be our Lemmata \ref{thm0+H} and \ref{thm0+regularH}. In order to do so, we have to prove a bound for the countable Hausdorff weight of the sequential Aleksandroff compactification of a disjoint union of sequential refinement of Hausdorff spaces. Notice that, if we simply were to use the Aleksandroff compactification, we would not have any guarantee that the compactification would have been Hausdorff.

\begin{theorem}\label{thm-cHw-grosso}
Let $\mathcal F=\{Y_\alpha :\alpha<\mu\}$ be an infinite family of sequentially compact Hausdorff spaces, $X'$ the disjoint union of all spaces in $\mathcal F$, and $X=X'\cup\{\Omega\}$ the sequential Aleksandroff compactification of $X'$. Then, $X$ is a sequentially compact Hausdorff space, and \[\mathrm{cHw}(X)\le \textstyle\sup_{\alpha<\mu}\mathrm{cHw}(Y_\alpha).\]
\end{theorem}

\begin{proof}
First, we note that $X$ is sequentially compact by Lemma \ref{seq-compactification-is-seq-compact}.
Then, we shall prove that $X$ is Hausdorff. Let $p,q\in X$. If $\Omega\notin\{p,q\}$, it is easy to show they get separated: if they belong to the same $\{\alpha\}\times Y_\alpha$, this follows from the Hausdorff property of $\{\alpha\}\times Y_\alpha$; if they do not, hence $p\in \{\alpha\}\times Y_\alpha$ and $q\in \{\beta\}\times Y_\beta$, then $\{\alpha\}\times Y_\alpha$ and $\{\beta\}\times Y_\beta$ are disjoint open neighbourhoods of $p$ and $q$. Otherwise, suppose $p\in \{\alpha\}\times Y_\alpha$ for some $\alpha<\mu$ and $q=\Omega$. Since $\mathrm{seq}(Y_\alpha)$ is sequentially compact, we have that $X-\{\alpha\}\times Y_\alpha$ is an open neighbourhood of $\Omega$ in $X$, that is disjoint from $\{\alpha\}\times Y_\alpha$, which is an open neighbourhood of $p$ in $X$.\medskip

    Next, we prove that $\mathrm{cHw}(X)\le\sup_{\alpha<\mu}\mathrm{cHw}(Y_\alpha)$. Let $Z\subseteq X$ be countable, and for each $\alpha<\mu$ we let $Z_\alpha\subseteq Y_\alpha$ be such that $\{\alpha\}\times Z_\alpha=Z\cap (\{\alpha\}\times Y_\alpha)$. For each $\alpha<\mu$, we now define $\mathcal C_{Z_\alpha}$ as:
    \begin{itemize}  
    \item the empty set, if $Z_\alpha=\emptyset$;
    \item a sequentially separating limit cover for $Z_\alpha\subseteq \mathrm{seq}(Y_\alpha)$, if $Z_\alpha\ne\emptyset$.
    \end{itemize}

    In addition, we are now in the position to define
    \[\mathcal C_Z^0:=\{\mathcal L(Z)-(\{\alpha\}\times Y_\alpha):\alpha<\mu\};\]
    \[\mathcal C_Z:=\mathcal C_Z^0\cup\textstyle\bigcup_{\alpha<\mu}\big\{\{\alpha\}\times V:V\in\mathcal C_{Z_\alpha}\big\}.\]
    It is clear that $|\mathcal C_Z|\le|\mathcal C_Z^0|+\sup_{\alpha<\mu}\mathrm{cHw}(\mathrm{seq}(Y_\alpha))$. Moreover, since $Z$ is countable, $\mathcal L(Z)-(\{\alpha\}\times Y_\alpha)$ is equal to $L(Z)$ for each $\alpha<\mu$ except countably many, hence $|\mathcal C_Z^0|\le\aleph_0$. This proves $|\mathcal C_Z|\le\aleph_0+\sup_{\alpha<\mu}\mathrm{cHw}(\mathrm{seq}(Y_\alpha))=\sup_{\alpha<\mu}\mathrm{cHw}(\mathrm{seq}(Y_\alpha))$.\medskip
    
    Given Lemma \ref{chwseq}, to prove the statement we only need to prove that $\mathcal C_Z$ is a sequentially separating limit cover for $Z\subseteq X$. Let $p\in\mathcal L(Z)$ and let $q\in X-\{p\}$. We need to show there exists $U\in\mathcal C_Z$ with $p\in U$ and $q\notin\mathrm{SeqCl}(U)$. If there exists $\alpha<\mu$ such that $q\in \{\alpha\}\times Y_\alpha$ and $p\notin \{\alpha\}\times Y_\alpha$, then we can choose $U=X-\{\alpha\}\times Y_\alpha$. If instead we have $q=\Omega$, then we fix $\alpha$ such that $p\in \{\alpha\}\times Y_\alpha$, and we can choose any $U=\{\alpha\}\times V$ such that $V\in\mathcal C_{Z_\alpha}$ and $p\in U$. Finally, we have the case where there exists $\alpha<\mu$ with $p,q\in \{\alpha\}\times Y_\alpha$, and in that case we use the fact that $\mathcal C_{Z_\alpha}$ is a sequentially separating limit cover for $Z_\alpha\subseteq \mathrm{seq}(Y_\alpha)$ to choose a $U=\{\alpha\}\times V$ with $V\in\mathcal C_{Z_\alpha}$, $p\in U$ and $q\notin \mathrm{SeqCl}(U)$.
\end{proof}

\begin{lemma}\label{thm0+H}
    Fix some cardinal $\kappa<\mathfrak s$. Then, if $\lambda<\mathfrak s$, the family of all sequentially compact Hausdorff spaces $X$ with $\mathrm{cHw}(X)\le \kappa$ is $\lambda$-ssc.
\end{lemma}

\begin{proof}
    For $\alpha<\lambda$, let $X^\alpha$ be a sequentially compact Hausdorff space, $\mathrm{cHw}(X^\alpha)<\kappa$, and let $x^\alpha$ be a sequence with values in $X^\alpha$. Let $X$ be the sequential Aleksandroff compactification of the disjoint union of all spaces in $\mathcal F:=\{\mathrm{seq}(X^\alpha):\alpha<\lambda\}$.\medskip
    
    Theorems \ref{thm-cHw-grosso} and \ref{thm-cHw-1} yield that $X$ is $\lambda$-ssc, since $\mathrm{cHw}(X)\le \kappa\lambda<\mathfrak s$. Theorem \ref{thm0.5new+} implies that $\mathcal F$ is $\lambda$-ssc, so $\{(X^\alpha):\alpha<\lambda\}$ is also $\lambda$-ssc.
\end{proof}

\begin{lemma}\label{thm0+regularH}
    Assume $\mathfrak s$ is regular. If $\lambda<\mathfrak s$, the family of all sequentially compact Hausdorff spaces  $X$ with $\mathrm{cHw}(X)<\mathfrak s$ is $\lambda$-ssc.
\end{lemma}

\begin{proof}
    For $\alpha<\lambda$, let $X^\alpha$ be a sequentially compact space, $\mathrm{cHw}(X^\alpha)<\mathfrak s$. Then
$\{\mathrm{cHw}(X^\alpha): \alpha<\lambda\}$ is a set of cardinals below $\mathfrak s$ with size less than $\mathfrak s$, so since $\mathfrak s$ is regular, $\kappa:=\sup\{w(X^\alpha): \alpha<\lambda\}<\mathfrak s$. The statement now follows from Lemma \ref{thm0+H}.
\end{proof}

Finally, we can sharpen our results on Hausdorff spaces further by considering whether the countable Hausdorff weight is reached.

\begin{theorem}\label{thm-cHw-grosso-reach}
Let $\mathcal F=\{Y_\alpha :\alpha<\mu\}$ be an infinite family of sequentially compact Hausdorff spaces, $X'$ the disjoint union of all spaces in $\mathcal F$, and $X=X'\cup\{\Omega\}$ the sequential Aleksandroff compactification of $X'$. Suppose that we have: \[\mathrm{cHw}(X)=\textstyle\sup_{\alpha<\mu}\mathrm{cHw}(Y_\alpha)\ge\mu .\]
Suppose moreover that $\mathrm{cof}(\mathrm{cHw}(X))>\aleph_0$. Finally, suppose that for each $\alpha<\mu$ we have either $\mathrm{cHw}(Y_\alpha)<\mathrm{cHw}(X)$ or that $Y_\alpha$ does not reach its countable Hausdorff weight. Then, $X$ does not reach its countable Hausdorff weight.
\end{theorem}
\begin{proof}
    Fix any countable subset $Z\subseteq X$, and construct $\mathcal C_Z$ and $\mathcal C_{Z_\alpha}$ for $\alpha <\mu$ as in the proof of Theorem \ref{thm-cHw-grosso}. Let $\Lambda:=\{\alpha<\mu:C_{Z_\alpha}\ne\emptyset\}$, and notice that $\Lambda$ is countable. Moreover, we have:
    \[|\mathcal C_Z|\le\aleph_0+\textstyle\sup_{\alpha\in\Lambda}|\mathcal C_{Z_\alpha}|\]
    Since each $|\mathcal C_{Z_\alpha}|$ is strictly smaller than $\mathrm{cHw}(X)$, either because $\mathrm{cHw}(Y_\alpha)<\mathrm{cHw}(X)$ or because $Y_\alpha$ does not reach its countable Hausdorff weight, we get that $|\mathcal C_Z|$ is bounded by the supremum of countably many cardinals that are strictly smaller than $\mathrm{cHw}(X)$. Since $\mathrm{cof}(\mathrm{cHw}(X))>\aleph_0$, we get the statement.
\end{proof}

The analogue of Lemma \ref{thm0+H} is the following:
\begin{lemma}
    Let $\lambda$ be a cardinal, $\lambda<\mathrm{cof}(\mathfrak s)$. The family of all sequentially compact Hausdorff spaces that have either countable Hausdorff weight strictly smaller than $\mathfrak s$, or countable Hausdorff weight equal to $\mathfrak s$ but not reached, is $\lambda$-ssc.
\end{lemma}
\begin{proof}
    Throughout this proof we shall assume $\lambda>\aleph_0$, as the case $\lambda\le\aleph_0$ is covered by Corollary \ref{1+}. For $\alpha<\lambda$, let $X^\alpha$ be a sequentially compact Hausdorff space, and let $x^\alpha$ be a sequence with values in $X^\alpha$. Suppose that for each $\alpha<\lambda$ we have either $\mathrm{cHw}(X^\alpha)<\mathfrak s$, or $\mathrm{cHw}(X^\alpha)=\mathfrak s$ and $X^\alpha$ does not reach its countable Hausdorff weight. Let $X$ be the sequential Aleksandroff compactification of the disjoint union of all spaces in $\mathcal F:=\{\mathrm{seq}(X^\alpha):\alpha<\lambda\}$.\medskip
    
    Theorems \ref{thm-cHw-grosso} yields $\mathrm{cHw}(X)\le\mathfrak s$. If $\mathrm{cHw}(X)<\mathfrak s$, then $X$ is $\lambda$-ssc by Theorem \ref{thm-cHw-1}. If instead $\mathrm{cHw}(X)=\mathfrak s$, Theorems \ref{thm-cHw-grosso-reach} and \ref{thm-cHw-2} still yield that $X$ is $\lambda$-ssc. Either way, Theorem \ref{thm0.5new+} implies that $\mathcal F$ is $\lambda$-ssc, so $\{(X^\alpha):\alpha<\lambda\}$ is also $\lambda$-ssc.
\end{proof}

\section{Results based on the Distributivity Number}\label{s5}

The answer given by Theorem \ref{thm0} and its analogues is not solving completely our question, since we would like to have simultaneous sequential compactness for all topological spaces, and not only for those having some bounded weight-like characteristic. We do so in this section, but we need to lower the bound from the splitting number $\mathfrak s$ to the distributivity number $\mathfrak h$.

\begin{definition}
A set $\mathcal A\subseteq[\omega]^\omega$ is said to be \emph{open} if it is downwards closed, i.e., for all $H\in\mathcal A$ and all $H'\in[\omega]^\omega$ with $H'\subseteq H$ it holds $H'\in\mathcal A$. Similarly, a set $\mathcal D\subseteq[\omega]^\omega$ is \emph{dense} if it meets every open set, so if for all $H\in[\omega]^\omega$ there is some $H'\in\mathcal D$ such that $H'\subseteq H$. The \emph{distributivity number} $\mathfrak h$ is then defined as the least cardinality of a family of dense open subsets of $[\omega]^\omega$ having empty intersection.
\end{definition}

It holds $\aleph_1\le \mathfrak h\le 2^{\aleph_0}$, and it is consistent relative to \textsf{ZFC} that $\aleph_1<\mathfrak h<2^{\aleph_0}$. For more informations about the distributivity number, we refer again to \cite{blass}.

\begin{theorem}\label{thm1}
If $\lambda<\mathfrak h$, any sequentially compact space $X$ is $\lambda$-ssc.
\end{theorem}

\begin{proof}
Consider sequences $x^\alpha$ with values in $X$, for $\alpha<\lambda$, and define
$$\mathcal D_\alpha:=\{H\in[\omega]^\omega:\ x^\alpha\upharpoonright H\ \text{converges}\}.$$

If $x^\alpha\upharpoonright H$ converges and $H'\in[\omega]^\omega$ is such that $H'\subseteq H$, then $x^\alpha\upharpoonright H'$ converges; and for all $H\in[\omega]^\omega$ there is some $H'\subseteq H$ with $H'\in[\omega]^\omega$ and such that $x^\alpha\upharpoonright H'$ converges, since $X$ is sequentially compact. $\mathcal D_\alpha$ is dense open, so since $\lambda<\mathfrak h$, there is some $H\in[\omega]^\omega$ with $H\in\mathcal D_\alpha$ for all $\alpha<\lambda$, hence $x^\alpha\upharpoonright H$ converges for $\alpha<\lambda$.
\end{proof}

As a corollary of Theorems \ref{thm1} and \ref{thm0.5new}, we also get the following:

\begin{corollary}\label{thm1+}
If $\lambda<\mathfrak h$, the family of all sequentially compact spaces is $\lambda$-ssc.
\end{corollary}

From Theorem \ref{thm1} and Lemma \ref{2}, it follows that $\mathfrak h\le\mathfrak s$, a known inequality (see \cite{blass}). This also gives a complete answer to the problem when assuming $\mathfrak h=\mathfrak s$.

\begin{corollary}\label{h=s}
If $\mathfrak h=\mathfrak s$, a nontrivial sequentially compact space is $\lambda$-ssc iff $\lambda<\mathfrak s$.
\end{corollary}

Thanks to Theorem \ref{thm1} and Lemma \ref{2} we can also see, as shown in the following Lemma \ref{h<cofs}, that $\mathfrak h\le \mathrm{cof}(\mathfrak s)$, which is another well-known inequality, a strengthening of the previously mentioned one (and indeed, a proper strengthening, as shown by Alan Dow and Saharon Shelah in \cite{dowshe}).

\begin{lemma}\label{h<cofs}
It holds $\mathfrak h\le\mathrm{cof}(\mathfrak s)$.
\end{lemma}

\begin{proof}
For $\alpha<\mathfrak s$, let $x^\alpha$ be a sequence with values in $\{0,1\}$ such that there is no $H\in[\omega]^\omega$ along which all sequences $x^\alpha$ converge. These sequences $x^\alpha$ exist thanks to Lemma \ref{2}. Now let $\lambda_\beta$ for $\beta<\mathrm{cof}(\mathfrak s)$ be cardinals with $\sup\{\lambda_\beta:\beta<\mathrm{cof}(\mathfrak s)\}=\mathfrak s$.\medskip

Thanks to Lemma \ref{3}, we know that for $\beta<\mathrm{cof}(\mathfrak s)$ the sequences $x^\alpha$ for $\alpha<\lambda_\beta$ converge along a common set. Let $\mathcal D_\beta$ be the set of all $H\in[\omega]^\omega$ such that $x^\alpha\upharpoonright H$ converges for all $\alpha<\lambda_\beta$: this is a dense open  subset of $[\omega]^\omega$. Now, if $\mathrm{cof}(\mathfrak s)<\mathfrak h$, then the intersection of all $\mathcal D_\beta$ for $\beta<\mathrm{cof}(\mathfrak s)$ would be not empty (containing $H$), hence the sequences $x^\alpha$ for all $\alpha<\mathfrak s$ converge along a common $H$, a contradiction.
\end{proof}

Our next goal is to prove that there is a sequentially compact space which is not $\mathfrak h$-ssc. This trivially follows from $\mathfrak h=\mathfrak s$, but we are looking for a \textsf{ZFC} proof.

\begin{lemma}\label{fund}
Let $\mathcal D\subseteq[\omega]^\omega$ be dense open. There is a sequentially compact space $X$ and a sequence $x$ with values in $X$ such that $\mathcal D=\{H\in[\omega]^\omega: x\upharpoonright H\ \text{converges}\}$.
\end{lemma} 

\begin{proof}
Let $\mathcal L$ be the family of all open subsets of $\mathcal D$, and let $X:=\omega\cup\mathcal L$.\medskip

We now endow $X'$ with a topology, defined as follows. A subset $C\subseteq X'$ is closed if and only if for all $H\in [\omega]^\omega$ and $\ell\in\mathcal L$ with $H\in\ell$ and $H\subseteq C$, it holds $\ell\in C$.\medskip

Let us check that it is a topology. Clearly $\emptyset$ and $X'$ are closed. Let $C_\alpha$ be closed for all $\alpha<\kappa$, any $\kappa$, and prove that $C:=\bigcap_{\alpha<\kappa} C_\alpha$ is closed. If $H\in[\omega]^\omega$ and $\ell\in\mathcal L$ are such that $H\in\ell$, then assume $H\subseteq C$. This implies $H\subseteq C_\alpha$ for all $\alpha<\kappa$, so since $C_\alpha$ is closed we get $\ell\in C_\alpha$ for all $\alpha<\kappa$, then $\ell\in C$, as desired.\medskip

To finish, let $C_0$ and $C_1$ be closed and let us prove that $C:=C_0\cup C_1$ is closed. Assume that $H\in[\omega]^\omega$, $\ell\in\mathcal L$, $H\in \ell$, and $H\subseteq C$. Define now $H_0:=H\cap C_0$ and $H_1:=H\cap C_1$. At least one of $H_0$ and $H_1$ is an infinite subset of $\omega$; without loss of generality, we may assume $H_0\in[\omega]^\omega$. Then since $H\in \ell$, $H_0\subseteq H$ is infinite and $\mathcal D$ is open, we get $H_0\in \ell$. Since $C_0$ is closed and $H_0\subseteq C_0$, we have $\ell\in C_0\subseteq C$.\medskip

Consider now the sequence $x$ with values in $X'$ defined as $x_n:=n\in\omega$. Thanks to how we defined the topology on $X'$, we get that, for every $H\in[\omega]^\omega$, the subsequence $x\upharpoonright H$ converges to $\ell\in\mathcal L$ if and only if $H\in\ell$. In particular, we get that
$$\{H\in[\omega]^\omega: x\upharpoonright H\ \text{converges}\}=\textstyle\bigcup_{\ell\in \mathcal L}\ell = \mathcal D.$$

 Let now $X:=X'\cup\{\Omega\}$ be the Aleksandroff compactification of $X'$. We want to show that $X$ is sequentially compact. Let $y$ be a sequence with values in $X$ and let us distinguish between some possible cases. Clearly if $y$ takes infinitely many times the same value, there is a subsequence of $y$ that converges. Hence, by restricting to a subsequence if necessary, we assume $y$ takes each value at most once, and never takes value $\Omega$. By restricting furthermore to a subsequence if necessary, we might also assume that if $n<m<\omega$ are such that $y_n\in\omega$ and $y_m\in\omega$, then $y_n<y_m$.\medskip
 
 If $y$ takes value in $\omega$ infinitely many times, then $H:=\{n<\omega: y_n\in\omega\}\in[\omega]^\omega$. If we let $Z:=\{y_n\in\omega: n\in H\}$, then $y\upharpoonright H=x\upharpoonright Z$. Since $\mathcal D$ is dense, there is $Z'\in[\omega]^\omega$ with $Z'\subseteq Z$ such that $x\upharpoonright Z'$ converges, and if $H':=\{n<\omega: y_n\in Z'\}$ we get that $y\upharpoonright H'=x\upharpoonright Z'$ converges, so $y$ has a convergent subsequence.\medskip
 
To finish the proof of the sequential compactness of $X$, assume $y$ takes values in $\omega$ only finitely many times. By eventually restricting to a subsequence, we may assume that $y$ only takes values in $\mathcal L$ and never takes the same value twice. But then, since the relative topology on $\mathcal L$ is the discrete one, we know that a subset of $\mathcal L$ is compact if and only if it is finite. Hence, if $U$ is an open neighbourhood of $\Omega$, the sequence $y$ is eventually contained in $U$, hence $y$ converges to $\Omega$, as desired.
\end{proof}

\begin{theorem}\label{h<s}
There is a sequentially compact space $X$ which is not $\mathfrak h$-ssc.
\end{theorem}

\begin{proof}
Let $\langle \mathcal D_\alpha: \alpha<\mathfrak h\rangle$ be a family of dense open subsets of $[\omega]^\omega$ that have empty intersection. For all $\alpha<\mathfrak h$, thanks to Lemma \ref{fund} we can find a sequentially compact $X^\alpha$ and a sequence $x^\alpha$ with values in $X^\alpha$ and $\mathcal D_\alpha=\{H\in[\omega]^\omega: x^\alpha\upharpoonright H\ \text{converges}\}$.\medskip

Since the intersection $\bigcap_{\alpha<\mathfrak h}\mathcal D_\alpha$ is empty, there is no $H\in[\omega]^\omega$ along which all the sequences $x^\alpha$ converge. Hence the family of all sequentially compact spaces is not $\mathfrak h$-ssc, so thanks to Theorem \ref{thm0.5new} there is a sequentially compact topological space $X$ which is not $\mathfrak h$-ssc, as we wanted.
\end{proof}

After a first draft of this article was completed, we learned --through a personal communication by Michael Hrušák, whom we thank deeply-- that our Theorems \ref{thm1} and \ref{h<s} are easily seen to be equivalent to results whose proofs can be found in Petr Simon's article \cite{simon}. Since our proofs are shorter, in another context, and use different techniques, we believe it is still worthwhile to present them here.\medskip

In the proof of Lemma \ref{fund}, we could take $\mathcal L$ to be the family of all open subsets of $\mathcal D$ which are closed under finite unions and maximal with these properties. If we do so, the sequentially compact space $X$ not $\mathfrak h$-ssc we get from Theorem \ref{h<s} is such that $w(X)\le 2^{\aleph_0}$. We can now ask if it is possible to find a sequentially compact space $X$ not $\mathfrak h$-ssc and such that $w(X)=\mathfrak s$. We find interesting to ask both if this is provable in \textsf{ZFC}, and if this is consistent together with the hypothesis $\mathfrak h<\mathfrak s<2^{\aleph_0}$.\medskip

An even more interesting question is: is it consistent with \textsf{ZFC} that the family of all sequentially compact spaces of weight less than $\mathfrak s$ is not $\lambda$-ssc, for some $\lambda<\mathfrak s$? A positive answer can hold only in a model where $\mathfrak s$ is singular (hence, where $\mathfrak h<\mathfrak s$).

\section*{Acknowledgements}

We thank Angelo Bella for providing precious feedback on various iterations of this work.\medskip

The first author acknowledges the MIUR Excellence Department Project awarded to
the Department of Mathematics, University of Pisa, CUP I57G22000700001.\medskip

The second author is a member of the Gruppo Nazionale per le Strutture Algebriche,
Geometriche e le loro Applicazioni (GNSAGA) of the Istituto Nazionale di Alta
Matematica (INdAM).

\end{document}